\documentclass[12pt]{amsart}

\usepackage{amsmath,amssymb,latexsym}
\usepackage{amsfonts,amstext,amsthm,amscd}
\usepackage{graphicx, color}
\usepackage{enumerate}
\usepackage{amsthm}
\theoremstyle{plain}

\usepackage{color}

\usepackage{pstricks-add}
\usepackage{pst-all}

\newtheorem{theorem}{Theorem}[section]

\newtheorem{proposition}[theorem]{Proposition}

\newtheorem{lemma}[theorem]{Lemma}

\newtheorem{remark}[theorem]{Remark}

\newcommand{\A}{\mathbb A}
\newcommand{\Af}{\A_f}
\newcommand{\Ann}{\mathrm{Ann}}

\newcommand{\C}{\mathbb C}
\newcommand{\Co}{\mathrm C}
\newcommand{\D}{\mathcal D}
\newcommand{\Da}{D^{\alpha}}
\newcommand{\Dom}{\mathrm {Dom}}
\newcommand{\B}{\mathcal B}

\newcommand{\F}{\mathcal F}

\newcommand{\T}{\mathcal T}

\newcommand{\lcm}{\mathrm{lcm}}

\newcommand{\N}{\mathbb N}

\newcommand{\Pp}{\mathbb P}

\newcommand{\Q}{\mathbb Q}
\newcommand{\Qp}{\Q_p}
\newcommand{\Qs}{\Q_S}
\newcommand{\R}{\mathbb R}

\newcommand{\Real}{\mathrm{Re}}
\newcommand{\Ss}{\mathbf{S} }
\newcommand{\Tt}{\mathbf{T} }

\newcommand{\Z}{\mathbb Z}
\newcommand{\Zp}{\Z_p}
\newcommand{\Zs}{\Z_S}
\newcommand{\Zprod}{\prod_{p\in S} \Z_p}
\newcommand{\Zz}{\widehat{\Z}}

\newcommand{\abs}[1]{\left\vert#1\right\vert}
\newcommand{\absp}[1]{\left\vert#1\right\vert_p}
\newcommand{\mTo}{\longmapsto}
\newcommand{\norm}[1]{\left\Vert#1\right\Vert}
\newcommand{\ord}{\mathrm{ord}}
\newcommand{\To}{\longrightarrow}

\date{\today}

\begin{document}

\title[A Heat Equation  on some Adic Completions of $\Q$ and Ultrametric Analysis]{A Heat Equation on  Adic Completions of $\Q$ and Ultrametric Analysis}

\author[V.A. Aguilar--Arteaga, M. Cruz--L\'opez and S. Estala--Arias]{Victor A. Aguilar--Arteaga$^*$, Manuel Cruz--L\'opez$^{**}$ and Samuel Estala--Arias$^{***}$}

\address{$^{*}$ Departamento de Matem\'aticas, CINVESTAV, Unidad Quer\'etaro, M\'exico} 
\email{aguilarav@math.cinvestav.mx}

\address{$^{**}$ Departamento de Matem\'aticas, Universidad de Guanajuato,
Jalisco S/N Mineral de Valenciana, Guanajuato, Gto.\ C.P.\ 36240, M\'exico} 
\email{manuelcl@ugto.mx}

\address{$^{***}$ Ciencias de la Computaci\'on, INAOE, M\'exico} 
\email{samuelea@ccc.inaoep.mx}



\begin{abstract}  The product ring $\Qs=\prod_{p\in S} \Qp$, where   $S$ is a finite set of prime numbers and $\Qp$ is the field of $p$--adic numbers, is a second countable, locally compact and totally disconnected topological ring. This work introduces  a natural  ultrametric on $\Qs$ that allows to define a pseudodifferential operator $\Da$ 
 and to study an abstract heat equation on the Hilbert space $L^2(\Qs)$.  
 The fundamental solution of this equation is a normal transition function of a Markov process on $\Qs$. As a result, the techniques developed here  provides a general framework of these problems on other related
ultrametric groups.
\end{abstract}

\maketitle

\section*{Introduction}
\label{introduction}

The theory of pseudodifferential operators over the field $\Qp$ of $p$--adic numbers has been deeply explored by several authors. Many  classical equations have been formulated on the $n$-dimensional version of this non--Archimedean space and their solutions have been relevant to some models of modern physics (see e.g., \cite{VVZ}, \cite{Koc}  and \cite{Zun1} and the references therein). In particular, several analogous heat equations on $\Qp^n$ are now well understood: the fundamental solutions of these equations give rise to  transition functions of  Markov processes on $\Qp^n$, which are  non--Archimedean counterparts to the classical Brownian motion. 

Pursuing a generalization, numerous Markov stochastic processes have been defined and elaborated on several different ultrametric and algebraic spaces, such as  the finite ad\`ele ring $\Af$, and on   general locally compact topological groups, such as the complete ad\`ele ring $\A$, as documented in \cite{CE1}, \cite{CE2}, \cite{BGPW}, \cite{VVZ}, \cite{Koc}, \cite{KKS}, \cite{KR}, \cite{TZ} and \cite{Zun}, among others. 

This article deals with  a  Markov process related to the fundamental solution  of a heat  equation  on the direct product ring $\Qs=\prod_{p\in S} \Qp$, where $S$ is a fixed finite set of finite prime numbers. 

The techniques developed here are different from the well known ones: they are geometrical and very simple. In addition, this method apply in the case of a selfdual, second countable, locally compact, Abelian, topological group $G$ with a selfdual filtration $\{H_n\}_{n \in \Z}$ by compact and open subgroups, such that the indices $[H_n:H_{n-1}]$ are uniformly bounded. The authors show how to construct  many examples  of these groups.  (See Section \ref{final_remarks} for details and also \cite{BGPW} for other general setting and techniques.)

The ring $\Qs$ is a second countable, locally compact, totally disconnected,  commutative, topological ring and contains, as a maximal compact and open subring, the direct product ring $\Zs=\prod_{p\in S} \Zp$, where $\Zp$ is the ring of $p$--adic integers. Furthermore, as a topological group, it has a Haar measure $d\mu$ normalised to be a probability measure on $\Zs$,
it is selfdual in the sense of Pontryagin and there exists an additive character $\chi(\cdot)$ such that $x \mapsto \chi(x \cdot)$ gives an explicit isomorphism. Using this identification, the  subgroup $\Zs$ coincides with its own annihilator.

Let $\psi(n)$ be the second Chebyshev function (see Section \ref{ultrametrica}). This function determines a ``symmetric''  filtration of $\Qs$ by open and closed subgroups: 
$$
e^{\psi(m)}\Zs \subset e^{\psi(n)} \Zs \subset \Zs \subset e^{\psi(-n)}\Zs \subset e^{\psi(-m)}\Zs   \quad (m  \geq n \geq 1).
$$
There is a unique additive invariant ultrametric $d$ on $\Qs$ such that it  has the filtration above as the set of balls centred at zero and the Haar measure of any ball is equal to its radius. 

The ultrametric $d$ leads to define a pseudodifferential operator $\Da$ on the Hilbert space $ L^2(\Qs)$, which is similar to the Taibleson operator on $\Q_p^n$ but  looks also as the Vladimirov operator on $\Qp$. The operator $-\Da$ is a positive selfadjoint unbounded operator and allows to state the abstract Cauchy problem on $L^2(\Qs)$ for the classical heat equation. This problem is well posed and   the normalization  property of the ultrametric $d$ allows to give classical bounds of the heat kernel $Z(x,t)$, using  properties of the Archimedean Gamma function, and to find and explicit solution of this problem.

The main result on the solution of the Cauchy problem reads as follows (see Section \ref{heat_equation} for details and also Theorem \ref{generalization}).

\textbf{Theorem \ref{maintheorem}:} If $f$ is a complex valued square integrable function on $\Qs,$ which belongs to the domain of $-\Da$, the Cauchy problem
\begin{equation*}
\begin{cases}
\frac{{\partial}u(x,t)}{{\partial t}} + \Da u(x,t) = 0,  \ x \in \Qs, \ t \geq 0,  \\
u(x,0) = f(x),
\end{cases}
\end{equation*} 
has a classical solution $u(x,t)$ determined by the convolution of $f$ with the heat kernel $Z(x,t)$. In addition, $Z(x,t)$ is the transition density of a time and space homogeneous Markov process which is bounded, right--continuous and has no discontinuities other than jumps.

The exposition is organized as follows. In Section \ref{p-adic_numbers} the results of the theories of $p$--adic numbers and Fourier analysis on $\Qp$ are collected. The description of the ring $\Qs$ and several important properties related to the non--Archimedean metric $d$ are considered in Section \ref{Sadic_ring}. Section \ref{function_spaces} presents the theory of Fourier analysis and
introduces the pseudodifferential operator $\Da$ defined on $\Qs$ and in Section \ref{heat_equation} the homogeneous heat equation is treated. Section \ref{markov_process} provides the description of the Markov process derived from the fundamental solution to the heat equation. The general Cauchy problem on $\Qs$ is presented in Section \ref{Cauchy_problem}, and finally, in Section \ref{final_remarks} different aspects on generalization of the theory presented here are analysed. 

\section[The field of $p$--adic numbers]{The field of $p$--adic numbers}
\label{p-adic_numbers}

This section outlines the basic elements of the theory of $p$--adic numbers $\Qp$ as well as  the relevant function spaces defined on $\Qp$ and some aspects of the Fourier analysis on this space. For a comprehensive introduction to these subjects we quote \cite{VVZ}. 

Let $\N=\{ 1,2, \ldots \}$ be the set of natural numbers and let $\Pp$ be the set of prime numbers.  Fix a prime number $p\in \Pp$. If $x$ is any nonzero rational number, it can be written uniquely as $x=\displaystyle{p^k\frac{a}{b}}$, with $p$ not dividing the product $ab$ and $k\in \Z$. The function
\[ 
\absp{x} := 
\begin{cases}             
p^{-k}    & \text{ if } x \neq 0 , \\
0 & \text{ otherwise,}
\end{cases}
\] 
is  a non--Archimedean absolute value on $\Q$. The \textsf{field of $p$--adic numbers} $\Qp$ is defined as the completion of $\Q$ with respect to the distance induced by $\absp{\cdot}$. 

Any nonzero $p$--adic number $x$ has a unique representation of the form
\[ p^{\gamma}\sum_{i=0}^{\infty}a_ip^i, \] 
where $\gamma=\gamma(x) \in \Z, \, a_i \in \{0,1,\ldots,p-1\}$  and $a_0 \neq 0$. The value $\gamma$, with $\gamma(0)=+\infty$, is called the \textsf{$p$--adic order of $x$}. Any series of the above form converges in the topology induced by the $p$--adic metric.

The fractional part of a $p$--adic number $x$ is defined by
$$
\{x\}_p = 
\begin{cases}
p^{\gamma} \sum_{i=0}^{-\gamma-1} a_i p^i & \text{ if } \gamma <0, \\
0 & \text { if } \gamma \geq 0.       
\end{cases}
$$

The field $\Q_p$ is a locally compact topological field. The unit ball 
\[ \Zp = \{\, x \in  \Q_p \, : \, \abs{x}_p \leq 1 \,\} \] 
is the \textsf{ring of integers} and  the maximal compact and open subring of $\Qp$. Denote by $dx$ the Haar measure of the topological Abelian group $(\Qp,+)$ normalized to be a probability measure on $\Zp$.

 The algebraic and  topological properties of the ring $\Zp$ and the field $\Qp$ can be  expressed, respectively, by an inductive   and a projective limit
\begin{equation} \label{projetive-inductive}
\Zp =\varprojlim _{l\in \N \cup \{0\}}  \Zp/ p^{l} \Zp, \qquad  \Qp = \varinjlim_{l\in \N \cup \{0\}} p^{-l} \Zp. 
\end{equation} 

 In addition to the limits above, to the ring $\Zp$ and the field $\Qp$, there corresponds, respectively, an infinite rooted  tree $\mathcal{T}(\Zp)$ with constant  ramification index $p$, and an extended tree $\mathcal{T}(\Qp)$ with constant ramification index $p$. The endpoints of these trees corresponds, respectively, to  $\Zp$ and $\Qp$.

A function $\phi:\Qp\To \C$ is  \textsf{locally constant} if for any $x\in \Qp$, there exists an integer $\ell(x)\in \Z$ such that
\[ \phi(x+y) = \phi(x), \quad \text{for all}\; y\in B_{\ell(x)}, \]
where $B_{\ell(x)}$ is the closed ball with centre at zero and radius $p^{\ell(x)}$.

The set of all locally constant functions of compact support on $\Qp$ forms a $\C$--vector space denoted by $\D(\Qp)$. The $\C$--vector space $\D(\Qp)$ is the \textsf{Bruhat--Schwartz space} of $\Qp$ and an element 
$\phi\in \D(\Qp)$ is called a \textsf{Bruhat--Schwartz function} (or simply a \textsf{test function}) on $\Qp$.

If $\phi$ belongs to $\D(\Qp)$ and its not zero everywhere, there exists a largest $\ell=\ell(\phi)\in \Z$ such that, for any $x\in \Qp$, the following equality holds
\[ \phi(x+y)=\phi(x), \text{ for all } y \in B_{\ell}. \] 
This number $\ell$ is called \textsf{the parameter of constancy} of $\phi$. 

Denote by $\D_k^{\ell}(\Qp)$ the finite dimensional vector space consisting of functions with  parameter of constancy is greater or equal than $\ell$ and whose  support is contained in $B_k$. 

A sequence $(f_m)_{m\geq 1}$ in $\D(\Qp)$ is a Cauchy sequence if there exist $k,\ell \in \Z$ and $M>0$ such that 
$f_m\in \D_k^\ell(\Qp)$ if $m\geq M$ and $(f_m)_{m\geq M}$ is a Cauchy sequence in $\D_k^\ell(\Qp)$. 	That is, 
$$
\D^{\ell}(\Qp) = \varinjlim_{k} \D^{\ell}_k(\Qp) \quad \text{ and } \quad  
\D(\Qp) =\varinjlim_{\ell} \D^{\ell}(\Qp).
$$

With this topology the space $\D(\Qp)$ is a complete locally convex topological algebra over $\C$. It is also a nuclear space because
$\D^{\ell}(\Qp)$  is the  inductive  limit of countable family of finite dimensional algebras and
$\D(\Qp)$ is the inductive  limit of countable family of nuclear spaces $\D^{\ell}(\Qp)$.

For each compact set $K\subset\Qp$, let $\D(K)\subset \D(\Qp)$ be the subspace of test functions whose support is contained in $K$. The space $\D(K)$ is dense in $\Co(K)$, the space of complex--valued continuous functions on $K$.

An additive character of the field $\Qp$ is defined as a continuous function $\chi:\Qp\To \C$ such that $\chi(x+y)=\chi(x)\chi(y)$ and $\abs{\chi(x)}=1$, for all $x,y \in \Qp$.  The function $\chi_p(x)=\exp(2\pi i\{x\}_p)$ defines a canonical additive character of $\Qp$ which is trivial on $\Zp$ and not trivial outside $\Zp$. In fact, all characters of $\Qp$ are given by $\chi_{p,\xi}(x)=\chi_p(\xi x)$ with $\xi \in \Qp$.

The Fourier transform of a test function $\phi \in \D(\Qp)$ is given by the formula
$$
\F_p\phi(\xi) = \widehat{\phi}(\xi) = 
\int_{\Qp} \phi(x) \chi_p(\xi x)dx, \qquad  (\xi \in \Qp). 
$$ 

The Fourier transform is a continuous linear isomorphism of the space $\D(\Qp)$ onto itself and the inversion formula holds:
\[ \phi(x) = \int_{\Qp} \widehat{\phi}(\xi) \chi_p(-x\xi) d\xi \qquad (\phi\in \D(\Qp)). \]

The Parseval -- Steklov equality reads as: 
$$
\int_{\Qp} \phi(x) \overline{\psi(x)} dx =\int_{\Qp} \widehat{\phi}(\xi) \overline{\widehat{\psi}(\xi)} d\xi,
\qquad (\phi,\psi \in \D(\Qp)).
$$

\begin{remark} From  expressions (\ref{projetive-inductive}) or the description of $\Qp$ as collection of  endpoints of the tree $\T(\Qp)$, it follows that the Hilbert space $L^2(\Qp)$ has a numerable Hilbert base, which is an analogous of a wavelet base, and therefore it is a  separable Hilbert space (see e.g. \cite{AKS}). 

\end{remark}

\begin{remark} The extended Fourier transform $\F:L^2(\Qp) \To L^2(\Qp)$ is an isometry of Hilbert spaces and the Parseval -- Steklov identity holds on $L^2(\Qp)$.
\end{remark}

\section[The $S$--adic ring of $\Q$]{The $S$--adic ring of $\Q$}
\label{Sadic_ring}

This section introduces the ring $\Qs$ and an ultrametric on $\Qs$ invariant under translations by elements of $\Qs$ and under multiplication by units of $\Zs$. We will describe $\Qs$ as  topological ring itself, however other  proofs of our statements can be given by looking at the product structure of $\Qs$  (see \cite{CE1} for these kind of construction, where it is done for the finite ad\`ele ring of $\Q$).

\subsection{The adic ring $\Qs$}
Fix a finite subset $S\subset \Pp$ and define $\Qs$ as the \textsf{direct product} 
\[ \Qs =   \prod_{p\in S}  \Qp. \] The Tychonoff topology and the componentwise operations provide $\Qs$ with a structure of a topological ring. The topological ring $\Qs$ is commutative, second countable, locally compact and totally disconnected. The maximal compact and open subring of $\Qs$ is the direct product ring $\Zs = \Zprod$.

The additive group $(\Qs, +)$ is a locally compact Abelian group and therefore it has a Haar measure $d\mu$  which can be normalized to be a probability measure on $\Zs$. The measure $d\mu$ can be expressed in terms of the measures $dx_p$ on the groups $(\Qp,+)$ as the direct product measure
\[ d\mu = \prod_{p\in S} dx_p. \] 

With regard to characters, there is a  canonical additive character $\chi$ on $\Qs$, which is trivial on $\Zs$ and not trivial outside $\Zs$, given by 
\[ \chi(x) = \prod_{p\in S} \chi_p(x_p), \qquad \big(x= (x_p)_{p\in S} \in \Qs \big), \]
where $\chi_p(x_p)$ is the canonical character of $\Qp$. Recall that these characters are given by $\chi_p(x_p)=e^{2\pi i\{x_p\}_p}$, where $\{x_p\}_p$ is the $p$--adic fractional part of $x_p$.

For $\xi\in \Qs$, the application
\[ \chi_{\xi}(x) = \chi(\xi \cdot x) = \prod_{p\in S} \chi_p(\xi_p x_p), \qquad 
\big(\xi= (\xi_p)_{p\in S} \in \Qs \big), \] 
defines a character on $\Qs$. Moreover, since $\Qs$ is a direct product of some $p$-adic fields, any arbitrary character on $\Qs$ has the form
$\chi_{\xi}$, for some $\xi\in \Qs$. Therefore $\Qs$ is a selfdual group with isomorfism given by $\xi \mapsto \chi_{\xi}$.

Recall that the annihilator of a compact and open subgroup $H$ of $\Qs$ is the set of characters that are trivial in $H$. From the expression of the characters on $\Qs$,  the relation $$\Ann_{\Qs}(B_n) = B_{-n},$$ where $\Ann_{\Qs}(B_n)$ is the annihilator of $B_n$ in $\Qs$, holds.  In particular, $\Zs$ coincides with its own annihilator.

\begin{remark}
For the scope of this work the relevant properties of  $\Qs$ are that it a topological ring which is second countable, locally compact, and totally disconnected. In addition, as a topological group, $\Qs$ is  selfdual. 
\end{remark}

\subsection{An ultrametric on $\Qs$}\label{ultrametrica}
Let us introduce an ultrametric $d$ on $\Qs$ compatible with its topology, making $(\Qs,d)$ a complete ultrametric space. The set of rational numbers $\Q$ is diagonally embedded in $\Qs$ with dense image. In this sense, the ring $\Qs$ can be thought of as an $S$--completion of $\Q$ (see \cite{Bro} for a description of this topology  on the ring of integer numbers $\Z$). 

We start by defining two arithmetical functions which are related to the set $S$ and similar to the second Chebyshev and von Mangoldt functions (see e.g., \cite{Apo}). For any  natural number $n$, write 
$$
\Lambda(n) = \Lambda_S(n) =
\begin{cases} 
\log p & \text{if}\; n=p^k \text{ for some } p\in S \text{ and integer } k\geq 1, \\ 
0 & \text{otherwise}. 
\end{cases} 
$$ 
Likewise, let $\psi(n)$ denote the function implicitly given by  
$$ 
e^{\psi(n)} = e^{\psi_S(n)}= \lcm \left\{\, p^l \leq n : p \in S , \, l \in \N \cup \{0\}  \,\right\} 
\qquad(n \in \N). 
$$

These arithmetical functions are related by the equations
$$
\psi(n) =\sum_{k=1}^{n} \Lambda(k) \quad \text{ and } \quad 
e^{\psi(n)} = \prod_{k=1}^{n} e^{\Lambda(k)} \qquad (n \in \N).
$$ 

For any integer number $n$, define
\begin{equation*}
\psi(n) =
\begin{cases} 
\frac{n}{\abs{n}} \psi(\abs{n}) & \text{ if } n\neq 0, \\
0  & \text{ if } n=0, 
\end{cases}
\end{equation*}
and 
\begin{equation*}
\Lambda(n)=
\begin{cases} 
\Lambda(n) & \text{ if } n > 0, \\
\Lambda(\abs{n-1})=\Lambda(\abs{n}+1)  & \text{ if } n\leq 0. 
\end{cases}
\end{equation*}

The relations between these functions on the natural numbers extend to the integers in the following way: for any integers $n > m$
$$
\psi(n) - \psi(m) = 
\sum_{k=m+1}^{ n} \Lambda(k) \quad \text{ and } \quad e^{\psi(n)}/e^{\psi(m)} 
= \prod_{k =m+1}^n e^{\Lambda(k)}. 
$$ 

The collection $\{e^{\psi(n)}\Zs \}_{n\in \Z}$ of compact and open subgroups is a neighbourhood base of zero for the Tychonoff topology on $\Qs$ and determines a filtration (see figure 1) 
\[ \{0\}\subset \cdots \subset e^{\psi(n)}\Zs \subset \cdots \subset \Zs \subset \cdots \subset e^{\psi(m)}\Zs \subset \cdots \subset \Qs \qquad \big( n > 0 > m \big), \]
satisfying the properties:
\begin{equation*}
\label{union-interseccionQl}
\bigcap_{n\in \Z} e^{\psi(n)}\Zs = \{0\}\quad \text{ and } \quad 
\bigcup_{n\in \Z} e^{\psi(n)}\Zs = \Qs.
\end{equation*}

\begin{remark} From the properties of the above filtration, the topology of the ring $\Qs$ is expressed by the inductive and projective limits
\[ \Qs = \varinjlim_{n\in \N} e^{\psi(-n)} \Zs, \quad \Zs =\varprojlim _{n\in \N} \Zs/ e^{\psi(n)} \Zs.   \] 
In addition, we have the identity
$$
e^{\psi(n)}\Zs= \prod_{p\in S}p^{\ord_p(e^{\psi(n)})}\Z_p,
$$ where $\ord_p(\cdot)$ is the $p$--adic order function on $\Qp$.
\end{remark}

\begin{figure}\label{filtrationofQs}
\begin{center}
\newrgbcolor{cccccc}{0.8 0.8 0.8}
\newrgbcolor{ttttff}{0.2 0.2 1}
\newrgbcolor{ffttqq}{1 0.2 0}
\newrgbcolor{fffftt}{1 1 0.2}
\newrgbcolor{qqzzzz}{0 0.6 0.6}
\newrgbcolor{qqffcc}{0 1 0.8}
\newrgbcolor{ccttqq}{0.8 0.2 0}
\newrgbcolor{qqqqcc}{0 0 0.8}
\newrgbcolor{zzttqq}{0.6 0.2 0}
\newrgbcolor{ttqqcc}{0.2 0 0.8}
\newrgbcolor{qqqqzz}{0 0 0.6}
\newrgbcolor{ccqqqq}{0.8 0 0}
\newrgbcolor{ttqqzz}{0.2 0 0.6}
\newrgbcolor{qqttzz}{0 0.2 0.6}
\newrgbcolor{ccfftt}{0.8 1 0.2}
\psset{unit=.10cm,algebraic=true,dotstyle=o,dotsize=3pt 0,linewidth=0.8pt,arrowsize=3pt 2,arrowinset=0.25}
\begin{pspicture*}(-34.59,-33.95)(38.2,32.11)
\pscircle[linecolor=cccccc,fillcolor=cccccc,fillstyle=solid,opacity=0.5](0,0){26}
\pscircle[linecolor=ttttff,fillcolor=ttttff,fillstyle=solid,opacity=0.75](13,0){13}
\pscircle[linecolor=ffttqq,fillcolor=ffttqq,fillstyle=solid,opacity=0.75](-13,0){13}
\pscircle[linecolor=fffftt,fillcolor=fffftt,fillstyle=solid,opacity=0.75](-17.61,5.22){6.03}
\pscircle[linecolor=qqzzzz,fillcolor=qqzzzz,fillstyle=solid,opacity=0.75](-6.17,1.36){6.04}
\pscircle[linecolor=blue,fillcolor=blue,fillstyle=solid,opacity=0.75](-15.23,-6.6){6.04}
\pscircle[linecolor=fffftt,fillcolor=fffftt,fillstyle=solid,opacity=0.75](7.76,4.6){6.03}
\pscircle[linecolor=qqffcc,fillcolor=qqffcc,fillstyle=solid,opacity=0.75](19.6,2.21){6.04}
\pscircle[linecolor=red,fillcolor=red,fillstyle=solid,opacity=0.65](11.63,-6.83){6.04}
\pscircle[linecolor=ccttqq,fillcolor=ccttqq,fillstyle=solid,opacity=0.75](-19.61,7.48){3.02}
\pscircle[linecolor=qqqqcc,fillcolor=qqqqcc,fillstyle=solid,opacity=0.75](-15.62,2.96){3.02}
\pscircle[linecolor=zzttqq,fillcolor=zzttqq,fillstyle=solid,opacity=0.75](-9.13,0.77){3.02}
\pscircle[linecolor=ttqqcc,fillcolor=ttqqcc,fillstyle=solid,opacity=0.75](-3.21,1.95){3.02}
\pscircle[linecolor=qqqqzz,fillcolor=qqqqzz,fillstyle=solid,opacity=0.75](-14.26,-3.74){3.02}
\pscircle[linecolor=red,fillcolor=red,fillstyle=solid,opacity=0.75](-16.19,-9.46){3.02}
\pscircle[linecolor=ccqqqq,fillcolor=ccqqqq,fillstyle=solid,opacity=0.75](5.5,6.59){3.01}
\pscircle[linecolor=ttqqzz,fillcolor=ttqqzz,fillstyle=solid,opacity=0.75](10.02,2.61){3.01}
\pscircle[linecolor=ccqqqq,fillcolor=ccqqqq,fillstyle=solid,opacity=0.75](16.74,1.25){3.02}
\pscircle[linecolor=qqttzz,fillcolor=qqttzz,fillstyle=solid,opacity=0.75](22.46,3.17){3.02}
\pscircle[linecolor=ttqqcc,fillcolor=ttqqcc,fillstyle=solid,opacity=0.75](12.22,-3.87){3.02}
\pscircle[linecolor=ccfftt,fillcolor=ccfftt,fillstyle=solid,opacity=0.7](11.03,-9.79){3.02}
\end{pspicture*}
\end{center}
\caption{The decomposition of $\Qs$ by the filtration $\{e^{\psi(n)}\Zz\}$ for $S=\{2,3\}$.}
\end{figure}

For any element $x\in \Qs$ define the \textsf{order} of $x$ as:
\begin{equation*}
\ord(x) := 
\begin{cases} 
\max \left\{\, n\, :\, x\in e^{\psi(n)}\Zs \, \right\} & \text{ if } x\neq 0, \\
\infty  & \text{ if } x=0. 
\end{cases}
\end{equation*}

Notice that this order satisfies the following properties:
\begin{enumerate}[$\bullet$]
\item $\ord(x)\in \Z \cup\{\infty\}$ and $\ord(x)=\infty$ if and only if $x=0$, 
\item $\ord(x+y)\geq \min \{ \ord(x),\ord(y) \} $ and  
\item it takes the values
$$
S^{\Z} = \left\{ \, -p^{l} \, :\,  l \in \N  \text{ and } p\in S \, \right\} \cup 
\left\{ \, p^{l}-1 \, :\,  l \in \N  \text{ and } p\in S \, \right\} \cup \left\{\infty \right\}.
$$
\end{enumerate}

The nonnegative function  $d:\Qs\times \Qs\To \R^+\cup \{ 0\}$ given by
\[ d(x,y) = e^{- \psi \big(\ord(x-y)\big)} \]
is an ultrametric on $\Qs$. 
This ultrametric $d$ takes values in the set 
$\{ e^{\psi(n)} \}_{n\in \Z}\cup \{ 0\}$, any ball $B_n$ centred at zero with radius $e^{\psi(n)}$ is precisely the subgroup 
\[ B_n = B(0,e^{\psi(n)}) = e^{-\psi(n)}\Zs \qquad (n\in \Z), \]
 and any \textsf{sphere} centred at zero and radius $e^{\psi(n)}$ is 
\[ S_n = S(0,e^{\psi(n)})=B_n \backslash B_{n-1}. \]

The norm induced by this ultrametric is given by 
$$
\norm{x}=e^{- \psi (\ord(x))}
$$ and $\norm{x}=e^{\psi(n)}$ if and only if $x \in S_n$.

It is worth to notice that the radius of any ball on $\Qs$ is equal to its Haar measure:
\[ \int_{B_n+y} dx = \int_{B_n} dx  =  \int_{e^{-\psi(n)}\Zs} dx =  e^{\psi(n)} \qquad ( y \in \Qs, \, n \in \Z). \]
Using this fact, the area of any sphere is given by 
\[ \int_{S_n+y} dx = \int_{S_n} dx  =  e^{\psi(n)} - e^{\psi(n-1)}\qquad ( y \in \Qs, \, n \in \Z). \]

\begin{remark}
\label{prime_ramification}
If $e^{\psi(n)} < e^{\psi(n+1)}$,  $B_{n+1}$ is strictly contained in $B_n$; otherwise, if $e^{\psi(n)} = e^{\psi(n+1)}$,  $B_{n+1}=B_n$. This behaviour is controlled by the function $\Lambda(n)$  and  there exists a unique increasing bijective function 
$\rho :\Z  \To S^{\Z}$, such that $e^{\psi(\rho(n))}$ is a strictly increasing function with $e^{\psi(\rho(0))}=1$. For this reason, in the sequel, we suppose that 
$e^{\psi(n)} < e^{\psi(n+1)}$ for any integer number $n$.
\end{remark}

\section[Function spaces and pseudodifferential operators on $\Qs$]{Function spaces and pseudodifferential operators on $\Qs$}
\label{function_spaces}

The relevant spaces of test functions as well as the basic facts of Fourier analysis on $\Qs$ are described in this section. It also introduces a pseudodifferential operator $\Da$ on $\Qs$. The description made here follows closely the account made in \cite{CE2} for the finite ad\`ele ring of $\Q$. A second point of view is obtained by looking at the product structure of $\Qs$.

The reader can consult these topics in the excellent books  \cite{VVZ}, \cite{AKS}, \cite{Igu}. The theory of general topological vector spaces can be found in    \cite{Sch}.
 
\subsection[Bruhat--Schwartz test functions  on $\Qs$]{Bruhat--Schwartz test functions  on $\Qs$}
\label{bruhat-schwartz}

The \textsf{Bruhat--Schwartz space} $\D(\Qs)$ is the space of locally constant functions on $\Qs$ with compact support. Being $\Qs$ a totally disconnected space, $\D(\Qs)$ has a natural topology which can be described by two inductive limits,
given any filtration by compact and open sets in $\Qs$.

Let us describe the topology of $\D(\Qs)$ using the ultrametric $d$. If $ \phi \in \D(\Qs)$, there exists a smallest 
$\ell=\ell_\phi\in \Z$ such that, for every $x\in \Qs$, 
\[ \phi(x+y)=\phi(x),\quad \text{ for all } \; y \in B_{\ell}= e^{-\psi(\ell)} \Zs. \]
This number $\ell$ is called \textsf{the parameter of constancy} of $\phi$.
The set of all locally constant functions on $\Qs$ with common parameter of constancy 
  $\ell\in \Z$ and support in $B_k$ forms a finite dimensional complex vector space of dimension $e^{\psi(\ell)}/e^{\psi(k)}$. Denote this space by $\D^{\ell}_k(\Qs)$. The topology of $\D(\Qs)$ is expressed by the inductive limits
  $$
  \D^{\ell}(\Qs) = \varinjlim_{k} \D^{\ell}_k(\Qs) \quad \text{ and } \quad  
  \D(\Qs) =\varinjlim_{\ell} \D^{\ell}(\Qs),
  $$ which essentially states that every finite dimensional vector space $\D^{\ell}_k(\Qs)$ is open in $\D(\Qs)$. Therefore, $\D(\Qs)$ is a complete locally convex topological algebra over $\C$ and a nuclear space.

Finally, for each compact subset $K\subset \Qs$, let $\D(K)\subset \D(\Qs)$ be the subspace of test functions with support on a fixed compact subset $K$. The space $\D(K)$ is dense in $\Co(K)$, the space of complex valued continuous functions on $K$. The space $\D(\Qs)$ is dense in $L^2(\Qs)$.

\subsubsection[Bruhat--Schwartz test functions  as a tensor product]{Bruhat--Schwartz test functions  as a tensor product}
\label{bruhat-schwartzastensor}
Since the Tychonoff topology on $\Qs$ is the box topology, any function $\phi\in \D(\Qs)$ can be written as a finite linear combination of elementary functions of the form
\[ \phi(x) = \prod_{p\in S} \phi_p(x_p), \qquad \big(x= (x_p)_{p\in S} \in \Qs \big), \] 
where each factor $\phi_p(x_p)$ belongs to the space of test functions $\D(\Qp)$.  

Moreover, from the finite group  identification
$$
B_l/B_k = \prod_{p \in S} B^p_{\ell_p}/B^p_{k_p},
$$ where $\ell_p=\ord_p(e^{\psi(l)})$ and $k_p=\ord_p(e^{\psi(k)})$, the  following identity between finite dimensional spaces holds:
  $$
  \D^{\ell}_k(\Qs)= \bigotimes_{p \in S} \D^{\ell_p}_{k_p}(\Qp).
  $$
  
Commuting the induced limits with the tensor products,  
it is possible to show that $\D^{\ell}(\Qs)$ corresponds to the algebraic and topological tensor product of nuclear spaces $\{ \D^{\ell}(\Qp) \}_{p \in S}$, that is to say
$$\D^{\ell}(\Qs)= \bigotimes_{p \in S} \D^{\ell_p}(\Qp).$$
Furthermore, $\D(\Qs)$ corresponds to  the algebraic  and   topological tensor product of  nuclear spaces  $\{ \D(\Qp) \}_{p \in S}$, viz.,
\[ \D(\Qs) \cong \bigotimes_{p\in S} \D(\Qp). \]

\subsection[The Fourier transform on $\Qs$]{The Fourier transform on $\Qs$}
\label{The Fourier transform on Qs}

The Fourier transform of $\phi\in \D(\Qs)$ is defined by
\[ \widehat{\phi}(\xi) = \F[\phi](\xi) = \int_{\Qs} \phi(x) \chi(\xi x) dx, \qquad (\xi\in \Qs). \]
In particular, the Fourier transform of any elementary function $\phi=\prod_{p\in S}\phi_p$ is given by  
\[ \F[\phi](\xi) = \prod_{p\in S} \int_{\Qp} \phi_p(x_p) \chi_p(\xi_p x_p) dx_p, \qquad \left(\phi(x) = \prod_{p\in S} \phi_p(x_p), \  \xi=(\xi_p)_{p\in S} \in \Qs \right).  \]

The following two integrals are of main importance  for the properties of the Fourier transform on the Bruhat-Shwartz space,  $\D(\Qs)$.

\begin{lemma}
\begin{equation*}
\int_{B_n} \chi(-\xi x)dx = 
\begin{cases} 
e^{\psi(n)} & \text{ if } \norm{\xi} \leq e^{-\psi(n)}, \\
0  & \text{ if } \norm{\xi} > e^{-\psi(n)}. 
\end{cases}
\end{equation*}
\end{lemma}

\begin{proof}
This follows from the equality, $\Ann_{\Qs}(B_n) = B_{-n}$ and the fact that in any compact topological group, the integral of a nontrivial character over the group is zero.
\end{proof}  

\begin{lemma}
\label{integral_sphere} 
For any $n\in \Z$,
\begin{equation*}
\int_{S_n} \chi(-\xi x)dx = 
\begin{cases} 
e^{\psi(n)} - e^{\psi(n-1)} & \text{ if } \norm{\xi} \leq e^{-\psi(n)}, \\
-e^{\psi(n-1)} & \text{ if } \norm{\xi} = e^{-\psi(n-1)}, \\
0  & \text{ if } \norm{\xi} \geq e^{-\psi(n-2)}. 
\end{cases}
\end{equation*}
\end{lemma}

\begin{proof}
The integral can be decomposed as
\[ \int_{S_n} \chi(-\xi x) dx = \int_{B_n} \chi(-\xi x) dx - \int_{B_{n-1}} \chi(-\xi x)dx, \]
and the last proposition provides the result.
\end{proof}

\begin{remark}  From the well known computations of analogous integrals on $\Qp$, another proof of the last lemmas can be derived directly.
\end{remark}

From   last propositions,  definitions of the spaces $\D^{\ell}_k(\Qs)$ and Fourier transform, it follows 
\[ \F: \D^{\ell}_k(\Qs)\To \D^{-k}_{-\ell}(\Qs). \]
Moreover,  the Fourier transform $\F$ is a continuous linear isomorphism from $\D(\Qs)$ into itself. 
It is worth to notice that this  property of the Fourier transform follows directly    by construction, because $\Ann_{\Qs} (B_n)=B_{-n}$.

\begin{remark} 
Recall that, in the case of $\Qp$, the Fourier transform $\F_p$ sends $\D^{\ell}_k(\Qp)$ into 
$\D^{-k}_{-\ell}(\Qp)$ and the identification of $\D^{\ell}_k(\Qs)$ with the tensor product, $ \otimes_{p \in S}\D^{\ell_p}_{k_p}(\Qp)$, of finite dimension spaces. Therefore, $ \F: \D^{\ell}_k(\Qs)\To \D^{-k}_{-\ell}(\Qs)$  and the Fourier transform gives and isomorphism $\D(\Qs)\cong \D(\Qs)$.

\end{remark}

From the description of $\Qs$ and $\Zs$, respectively, as an inductive  and a projective limit or from the description of $\Qs$ as the endspace of a regular infinite tree it follows that the Hilbert space $L^2(\Qs)$ has a numerable Hilbert base which is a counterpart of a wavelet bases. Therefore,   $L^2(\Qs)$ is a separable Hilbert space. In addition, the Fourier transform
\[ \F: L^2(\Qs) \To L^2(\Qs) \]
is an isometry.

\begin{remark}Notice that $L^2(\Qs) \cong \bigotimes_{p\in S} L^2(\Qp)$, where the tensor denotes the Hilbert tensor product, because each $L^2(\Qp)$ is a separable Hilbert space, then \[ \F: L^2(\Qs) \To L^2(\Qs) \] is an isometry.
\end{remark}

\subsection[Lizorkin space of test functions  on $\Qs$]{Lizorkin space of test functions  on $\Qs$}
\label{lizorkin}

Another  space of test functions which is useful in the study of the heat equation is the following: Let $\Psi(\Qs)$ be the space of test functions which vanish at zero
\begin{equation*}
\label{Lizorkin-imagen}
\Psi(\Qs) = \left\{\, f\in \D(\Qs) : f(0)=0  \,\right\}. 
\end{equation*}

This means that for any element $f\in \Psi(\Qs)$ there exists a ball $B_n$ with centre at zero and radius $e^{\psi(n)}$ such that $f_{\mid B_n} \equiv 0$. The image of $\Psi(\Qs)$ under the Fourier transform is the space
\begin{equation*}
\label{Lizorkin}
\Phi(\Qs) = \left\{\,  g : g=\F[f], \, f \in \Psi(\Qs) \,\right\} \subset \D(\Qs)
\end{equation*}
called the \textsf{Lizorkin space of test functions of the second kind}. The space $\Phi(\Qs)$ is nontrivial and, as a subspace of $\D(\Qs)$, it has the subspace topology which makes it a complete topological vector space. 

\subsection[Pseudodifferential operators on $\Qs$]{Pseudodifferential operators on $\Qs$}
\label{pseudodifferential_operators}

For any  $ \alpha>0$ consider the   function $\norm{\cdot}^\alpha:\Qs \To \R_{\geq 0}$. The pseudodifferential operator, 
$\Da:\Dom(\Da) \subset L^2(\Qs)\To L^2(\Qs)$, defined by the formulae
\[ \Da f = \F^{-1}_{\xi \to x}[\norm{\xi}^\alpha \F_{x \to \xi}[f]], \]
for any $f$ in the dense domain
\begin{equation*}
\Dom(\Da) = \left\{\, f \in L^2(\Qs) :\norm{\xi}^\alpha\widehat{f} \in L^2(\Qs) \,\right \}, 
\end{equation*}
is called a \textsf{pseudodifferential operator with symbol} $\norm{\xi}^\alpha$.

 In other words, if we consider the multiplicative operator  $m^{\alpha}:\Dom(m^{\alpha}) \subset L^2(\Qs)\To L^2(\Qs)$ given by
 $$m^{\alpha}( f) (\xi)=\norm{\xi}^\alpha f(\xi),$$ with (dense) domain 
\begin{equation*}
\Dom(m^{\alpha}) = \left\{\, f \in L^2(\Qs) : \norm{\xi}^\alpha f \in L^2(\Qs)  \,\right \}, 
\end{equation*}
the unbounded operator $\Da$ with domain $\Dom(\Da)$ is the unique unbounded operator  such that the diagram
\begin{equation}\label{DiagramaConmutativo}
\begin{CD}
L^{2}(\Qs) @>\F>>L^{2}(\Qs) \\
@V{\Da}VV @VV{m^{\alpha}}V\\
L^{2}(\Qs) @>\F>> L^{2}(\Qs) 
\end{CD}
\end{equation}
commutes. 

Therefore, several properties of $A$ can be translated into the multiplicative operator $m^{\alpha}$: $\Da$ is a positive selfadjoint unbounded operator. 
Moreover, the commutative property or the diagram above means that
$\Da$ is diagonalized by the unitary Fourier transform $\F$.

Since $\Da$ is a positive and selfadjoint operator, its  spectrum  is contained in $[0,\infty)$.
The characteristic equation 
\[ \Da f = \lambda f \qquad \big(f \in  L^2(\Qs)\setminus \{0\} \big)\]
is solved by applying it the Fourier transform  and solving the resulting equation,  $(m^{\alpha} - \lambda )\widehat{f}=0$, or 
\[ (\norm{\xi}^{\alpha} - \lambda )\widehat{f}(\xi)=0. \]

If $\lambda \in \{e^{\alpha \psi(n)}\}_{n\in \Z}$, the characteristic function of the sphere $S_n$, $\Delta_{S_n}$, is a solution of the characteristic equation of the multiplicative operator $m^\alpha$. Otherwise, if $\lambda \notin \{e^{\alpha \psi(n)}\}_{n\in \Z}$, the function $\norm{\xi}^{\alpha} - \lambda$ is bounded from below and $\lambda$ is in the resolvent set of the multiplicative operator. Since the Fourier transform is unitary, the point spectrum of $\Da$ is the set $\{e^{\alpha \psi(n)}\}_{n\in \Z}$, with corresponding eigenfunctions  
$\{ \F^{-1}(\Delta_{S_n}) \}_{n\in \Z}$. Finally, $\{0\}$ forms part of  the spectrum as a limit point. Each eigenspace is infinite dimensional and there exist a well defined wavelet base which is also made of eigenfunctions (see \cite{CE2}) .

\section[A heat  equation on $\Qs$]{A heat equation on $\Qs$}
\label{heat_equation}

This section presents the analysis of the homogeneous heat equation on $\Qs$ related to  the pseudodifferential operator $\Da$ introduced in Section \ref{pseudodifferential_operators}.  This study is  done by treating the heat equation as an evolution equation on the Hilbert space $L^2(\Qs,d\mu)$  of square integrable complex valued functions on $\Qs$. The properties of general evolution equations on Banach spaces can be found in \cite{Paz} and \cite{EN}. 

For any function $f(x) \in \Dom(\Da)$, the pseudodifferential equation of the form
\begin{equation}
\label{HeatEquation}
\begin{cases}
\frac{{\partial}u(x,t)}{{\partial t}} + \Da u(x,t)=0,  \ x \in \Qs, \ t > 0 ,  \\
u(x,0)=f(x),
\end{cases}
\end{equation} 
is a non--Archimedean counterpart to the Archimedean homogeneous heat equation. 

In the $L^2(\Qs)$ context, a function $u:\Qs\times \R \To \C$ is called  a \textsf{classical solution} of the Cauchy problem  if: 
\begin{enumerate}[a.]
\item $u:[0,\infty) \To L^2(\Qs)$ is a continuously differentiable function,
\item $u(x,t)\in \Dom(\Da)$, for all $t\geq 0$, (in particular  
$f\in \Dom(\Da)$) and,
\item $u(x,t)$ is a solution of the initial value problem. 
\end{enumerate}

This problem is called here  an abstract Cauchy problem and will be referred as   problem (\ref{HeatEquation}). This problem is well posed and  its concept of solution is well understood from the theory of semigroups of linear operators.  This solution is described in the following section.

\subsection[Semigroup of operators]{Semigroup of operators}
\label{semigroup_operators}

From  the Hille--Yoshida Theorem, to the positive selfadjoint operator $-\Da$ there corresponds a strongly continuous contraction semigroup 
$$\Ss(t)=\exp(-t\Da):L^2(\Qs)\To L^2(\Qs) \qquad(t \geq 0),$$ with infinitesimal generator $-\Da$. 

It follows that $\{\Ss(t)\}_{t\geq 0}$ has the following properties

\begin{itemize}

\item as a function of $t$, $\Ss(t)$ is strongly continuous,
\item for $\geq 1$, $\Ss(t)$ is a bounded operator with  operator norm less than one,
\item  $\Ss(0)$ is the identity operator in $L^2(\Qs)$, i.e.  $\Ss(0)(f)=f$, for all $f \in L^2(\Qs)$,
\item it has the semigroup property: $\Ss(t)\cdot \Ss(s)=\Ss(t+s)$ , 
\item 
 if $f\in \Dom(-\Da)$, then  $\Ss(t)f \in \Dom(-\Da)$ for all $t \geq 0$,
the $L^2$ derivative $\frac{d}{dt} \Ss(t)f$ exists, is continuous for $t \geq 0$, and is  given by
\[  \frac{d}{dt} \Ss(t)f \Big|_{t=t_0^+} = -\Da \Ss(t)f =  -\Ss(t)\Da f\qquad  (t_0 \geq 0 ).   \]
\end{itemize}

All these means that $\Ss(t)f$ is a classical solution of the heat equation (\ref{HeatEquation}) with initial condition $f\in \Dom(\Da)$  .

For any initial data $f\in  \Dom(\Da)$, the Fourier transform can be applied to equation (\ref{HeatEquation}), in the spatial variable $x$, in order to get the abstract Cauchy problem:
\begin{equation}\label{ACP3}
\begin{cases}
\widehat{u}_t(\xi,t) + \norm{\xi}^{\alpha} \widehat{u}(\xi,t) = 0, \quad \xi\in \Qs, t \geq 0,  \\
\widehat{u}(\xi,0) = \widehat{f}(\xi), \quad (\widehat{f} \in \Dom(m^\alpha)). 
\end{cases}
\end{equation}

The classical solution of this problem is the strongly continuous  contraction semigroup $
\exp(-tm^\alpha) :L^{2}(\Qs) \To L^{2}(\Qs)$ given by $$  f(\xi) \mapsto f(\xi) \exp(-t \norm{\xi}^\alpha).
$$

Furthermore, from the commutative diagram (\ref{DiagramaConmutativo}),  definitions of the infinitesimal generators of $\Ss(t)$ and $\exp(-tm^\alpha)$ and the abstract Cauchy problems (\ref{HeatEquation}), (\ref{ACP3})  and  the fact that the Fourier transform is unitary, for $t\geq 0$, the diagram

\begin{equation}
\begin{CD}
L^{2}(\Qs) @> \F >> L^{2}(\Qs) \\
@V{\Ss(t)}VV @VV\exp(-tm^\alpha)V\\
L^{2}(\Qs) @> \F >> L^{2}(\Qs) \notag
\end{CD}
\end{equation}
commutes. That is to say, $$\Ss(t)=\F^{-1}\exp(-tm^\alpha)\F.$$

\subsection[The heat kernel]{The heat kernel}
\label{heat_kernel}
The theoretical solution of the heat equation shown above can be found explicitly by introducing heat kernel:
\begin{align}
\label{HeatKernel}
Z(x,t) &= \F^{-1}\big(\exp(-t\norm{\xi}^{\alpha})\big) = 
\int_{\Qs} \chi(-x\xi) \exp(-t\norm{\xi}^{\alpha}) d\xi \notag.
\end{align}

We estimate the heat kernel using some properties of the Archimedean Gamma function $\Gamma(z)$.
Recall that $\Gamma(z)$ is a meromorphic function on the complex plane with simple poles at the nonpositive integers, satisfies the functional equation $\Gamma(z+1) = z\Gamma(z)$ and admits the integral representation 
\[ \Gamma(z) = \int_0^{\infty} s^{z-1} e^{-s} ds, \]
in the halfplane $\Real(z) > 0$. 

Putting the integral representation of the Archimedean Gamma function into its functional equation, we obtain
\begin{align*}
\Gamma(z+1)	&= z \int_0^{\infty} s^{z-1} e^{-s} ds \\
						&= \int_0^{\infty} e^{-s^{1/z}} ds,
\end{align*}
which is convergent in the halfplane $\Real(z) > 0$.

The first estimate of the heat kernel is given in the following (compare to \cite{Koc}, Lemma 4.1, pp. 134):

\begin{lemma}
\label{heat-estimatesI} 
For any $t>0$ and $\alpha >0$, the function $\xi \mTo \exp(-t\norm{\xi}^{\alpha})$
is integrable over $\Qs$ and consequently $Z(x,t)$ is well defined for any $t>0$ and $x\in \Qs$.
Furthermore, for any $t>0$ and $\alpha >0$,  the heat kernel $Z(x,t)$  satisfies the inequality
\[ \abs{Z(x,t)} \leq  C t^{-1/\alpha}, \qquad (x\in \Qs), \]
where $C$ is a constant depending on $\alpha$.
\end{lemma}

\begin{proof}
Since the Haar measure of any ball is equal to its radius, we obtain
\begin{align*}
\int_{\Qs} \exp(-t\norm{\xi}^{\alpha}) d\xi
&= \sum_{n=-\infty}^{\infty} \int_{S_n} \exp(-t\norm{\xi}^{\alpha}) d\xi \\
&= \sum_{n=-\infty}^{\infty} \exp(-te^{\alpha\psi(n)}) \big(e^{\psi(n)}-e^{\psi(n-1)}\big) \\
&<\int_0^{\infty} \exp(-ts^{\alpha}) ds \\
&=t^{-1/\alpha}\Gamma(1/\alpha+1). 
\end{align*}
Since $Z(x,t)	=    \int_{\Qs} \chi(-x\xi) \exp(-t\norm{\xi}^{\alpha}) d\xi $,
this also  proves the second assertion with $C=\Gamma(1/\alpha+1)$.

\end{proof}

\begin{proposition}\label{freeHeatKernel} 
The heat kernel $Z(x,t)$ is a positive function for all $x $ and $t>0 $. In addition 

\begin{align*}
Z(x,t)  =\sum_{\substack{n\in \Z \\ e^{\psi(n)} \leq \norm{x}^{-1}}} e^{\psi(n)} 
\left\{\exp(-t e^{\alpha\psi(n)}) - \exp(-t e^{\alpha\psi(n+1)}) \right\} \notag. 
\end{align*}

\end{proposition}

\begin{proof} 
Using Proposition \ref{integral_sphere}, if $\norm{x}=e^{-\psi(m)}$, then
\begin{align*}
Z(x,t) &= \sum_{n=-\infty}^{\infty} \int_{S_n} \chi(-x \xi) \exp(-t\norm{\xi}^{\alpha}) d\xi \\
& = \sum_{n=-\infty}^{\infty} \exp(-t e^{\alpha\psi(n)}) \int_{S_n} \chi(-x \xi) d\xi \notag \\
& = \sum_{n=-\infty}^{m+1} \exp( -t e^{\alpha\psi(n)}) \int_{S_n} \chi(-x \xi) d\xi \notag \\ 
& = -\exp(-t e^{\alpha\psi(m+1)}) e^{\psi(m)} + 
\sum_{n=-\infty }^{m}\exp(-t e^{\alpha\psi(n)}) (e^{\psi(n)} - {e^{\psi(n-1)}})\\
& = \sum_{n=-\infty}^{m} 
e^{\psi(n)} \left\{\exp(-t e^{\alpha\psi(n)}) - \exp( -t e^{\alpha\psi(n+1)}) \right\} \notag \\
& =\sum_{\substack{n\in \Z \\ e^{\psi(n)} \leq \norm{x}^{-1}}} e^{\psi(n)} 
\left\{\exp(-t e^{\alpha\psi(n)}) - \exp(-t e^{\alpha\psi(n+1)}) \right\} \notag. 
\end{align*}

This implies that $Z(x,t)$ is a positive function for all $x $ and $t>0 $. 

\end{proof}

\begin{remark} It is important to notice that the expression of the heat kernel in Proposition \ref{freeHeatKernel}  does not depend on the algebraic structure of $\Qs$: it depends only on the values of the second Chebyshev function related to $S$. 
\end{remark}

\begin{lemma}
\label{heat-estimatesII} 
For any $t>0$, $\alpha >0$ and $x\in \Qs$, the heat kernel $Z(x,t)$ is positive and satisfies the inequality
\[ Z(x,t) \leq  C t\norm{x}^{-\alpha-1}, \qquad (x\in \Qs, t>0), \]
where $C$ is a constant depending on $S$ and $\alpha$.
\end{lemma}

\begin{proof}

In order to prove the second assertion we proceed as follows.
From the inequality 
$1-e^{-s}\leq s$, valid for $s\geq 0$, it follows the inequality 
\begin{align*}
Z(x,t) &\leq \norm{x}^{-1}\sum_{\substack{n\in \Z \\ e^{\psi(n)} \leq \norm{x}^{-1}}} 
\left\{ \exp(-t e^{\alpha\psi(n)}) - \exp(-t e^{\alpha\psi(n+1)}) \right\}  \\
 &\leq \norm{x}^{-1}\big (1 - \exp(-t e^{\alpha\psi(m+1)})\big) \leq t \norm{x}^{-1} e^{\alpha\psi(n+1)}\\
&= te^{\alpha\psi(m)} e^{\alpha\Lambda(m+1)} e^{\psi(m)} \\
&= te^{(\alpha+1) \psi(m)} e^{\alpha\Lambda(m+1)} \\
&= t\norm{x}^{-\alpha-1} e^{\alpha\Lambda(m+1)}.
\end{align*}

As a result
\[ Z(x,t) \leq  C t\norm{x}^{-\alpha-1}, \qquad (x\in \Qs, \, t>0), \]
where $C=\max_{p\in S} \{p^{\alpha}\}$. 
\end{proof}

\begin{remark}
\label{ramification_index}
The constant involved in the last result depends on  the growing behaviour of the group index $\abs{B_n/B_{n-1}}$ of two consecutive balls centred at zero, which is a uniformly bounded quantity.  
\end{remark}

The following proposition exhibits the relevant heat kernel estimate.

\begin{proposition} \emph{\textsf{(Heat kernel estimates)}}
\label{heat-estimates} 
\[ Z(x,t) \leq Ct(t^{1/\alpha} + \norm{x})^{-\alpha - 1},\qquad (x\in \Qs, t>0). \]
\end{proposition}

\begin{proof}
This is a straight  consequence of  Lemma \ref{heat-estimatesI} and Lemma \ref{heat-estimatesII} (see \cite{Koc}, Lemma 4.1, pp. 134).
\end{proof}

\begin{proposition}
\label{heat-distribution}
The heat kernel satisfies the following properties:
\begin{itemize}
\item It is the distribution of a probability measure on $\Qs$, i.e. $Z(x,t)\geq 0$ and
\[ \int_{\Qs} Z(x,t)dx = 1, \] 
for all $t>0$. 
 
\item It converges to the Dirac distribution as $t$ tends to zero:
\[ \lim_{t \to 0} \int_{\Qs} Z(x,t) f(x) dx = f(0), \] 
for all $f\in \D(\Qs)$.

\item It has the Markovian property:
\[ Z(x,t+s) = \int_{\Qs} Z(x-y,t) Z(y,s) dy. \]
\end{itemize}
\end{proposition}

\begin{proof} 
From Proposition \ref{heat-estimates}, it follows that $Z(x,t)$ is in $L^1(\Qs)$, for any $t>0$. Indeed,
\begin{align*}
\int_{\Qs} Z(x,t)dx &=\int_{\Zs} Z(x,t)dx +\int_{\Qs \backslash \Zs} Z(x,t)dx  \\
& \leq C_1 + C_2\int_{\Qs \backslash \Zs} \frac{1}{\norm{x}^{1+\alpha}}dx\\
& \leq C_1 + C_2\int_1^\infty \frac{1}{s^{1+\alpha}}ds.
\end{align*}

Being $\exp(-t\norm{\xi}^{\alpha})$ a continuous function on $\xi$, the Fourier inversion formula implies
\[ \int_{\Qs} Z(x,t) dx = 1. \]

Using this equality, the fact that $f\in \D(\Qs)$ is a locally constant function of  compact support and Proposition \ref{heat-estimates}, we conclude that
\[ \lim_{t \to 0} \int_{\Qs} Z(x,t) f(x) dx = f(0). \]

The Markovian property follows from the Fourier inversion formula and the related property  of the exponential function.
\end{proof}

\begin{proposition}
\label{convergence} 
For any $t>0$ and $\alpha,\beta >0$, the function
\[ \xi \longmapsto \norm{\xi}^\beta \exp(-t\norm{\xi}^{\alpha}) \]
is integrable over $\Qs$. Therefore, $Z(x,t)$ is smooth with respect to $t$ and the derivative 
\[ 
\frac{d}{dt} Z(x,t) = 
-\int_{\Qs} \norm{\xi}^{\alpha} \chi(-x\xi) \exp(-t \norm{\xi}^{\alpha}) d\xi  \qquad( t >0 )
\]
is convergent. Furthermore, $Z(x,t)$ is uniformly continuous on $t$, i.e. $Z(x, t) \in C\big((0, \infty), C(\Qs )\big).$
\end{proposition}

\begin{proof}
The proof is similar to the one in 
Lemma \ref{heat-estimatesI}:
\begin{align*}
\int_{\Qs} \norm{\xi}^\beta \exp( -t\norm{\xi}^{\alpha}) d\xi & = 
\int_{\Zs} \norm{\xi}^\beta \exp( -t\norm{\xi}^{\alpha}) d\xi +\int_{\Qs\backslash \Zs} \norm{\xi}^\beta \exp( -t\norm{\xi}^{\alpha}) d\xi \\ 
&<C+\sum_{n=1}^{\infty} e^{\beta\psi(n)} \exp(-te^{\alpha\psi(n)}) 
\big(e^{\psi(n)}-e^{\psi(n-1)}\big) \\
& \leq C+ \int_0^{\infty} s^\beta e^{-ts^{\alpha}}ds \\
 & \leq C+ \frac{\Gamma(\frac{\beta +1}{\alpha})}{\alpha t^{\frac{\beta+1}{\alpha}}} .
\end{align*}

 Finally, the last assertion follows from the mean value theorem.

\end{proof}

 \subsection[The classical solution of the Heat equation]{The classical solution of the heat equation} 


Given $t> 0$  define the operator $\Tt(t):L^2(\Qs) \To L^2(\Qs)$ by  the convolution with the Heat kernel
\[\Tt(t)f(x) = Z(x,t) * f(x), \qquad(f\in L^2(\Qs)), \]  and let $\Tt(0)$ be the identity operator.
From Proposition \ref{heat-distribution} and Young's inequality   the family of operators $\{\Tt(t)\}_{t \geq 0}$  is a $\Co_0$--semigroup.

Let us show  that the semigroup $\{\Tt(t)\}_{t \geq 0}$ gives the  solution of the heat equation (\ref{HeatEquation}) with initial data $f\in \Phi(\Qs)$. Recall that, for any  $f\in  \Phi(\Qs) \subset \Dom(\Da)$, we have  $\widehat{f}(\xi) \in \Psi(\Qs) \subset \Dom(-m^\alpha)$. Notice that   the abstract cauchy problem (\ref{ACP3}), 
with initial data $\widehat{f}(\xi) \in \Psi(\Qs)$, has a unique solution $\widehat{u}(\xi, t)$, such that $\widehat{u}(\cdot,t)$ belongs to $\Psi(\Qs)$ for any  $t\geq 0$, given by
\begin{equation*}
\widehat{u}(\xi,t) = \widehat{f}(\xi)\exp(-t\norm{\xi}^{\alpha}).
\end{equation*}

Therefore, there is a unique  solution  to the heat equation (\ref{HeatEquation}) with initial condition $f\in \Phi(\Qs)$, such that  $u(\cdot,t)\in \Phi(\Qs)$ for $t\geq 0$, given by
\begin{align*}
u(x,t) &= \F^{-1} \big( \widehat{f}(\xi) \exp(-t\norm{\xi}^{\alpha}) \big)\\
&= \int_{\Qs} \widehat{f}(\xi) \exp(-t\norm{\xi}^{\alpha}) \chi(-x\xi) d\xi \\
& =Z(x,t)*f(x).
\end{align*}

The main theorem of the diffusion equation is the following.

\begin{theorem}\label{maintheorem} 
Let $\alpha>0$ and let $\Ss(t)$ be  the $\Co_0$--semigroup generated by the operator $-\Da$. The operator $\Ss(t)$ coincides for each $t\geq 0$ with the operator $\Tt(t)$ given above. In other words, the  solution of the abstract Cauchy problem (\ref{HeatEquation}) is given by $u(x,t)=Z(x,t) * f(x)$, for $t\geq 0$ and $f \in \Dom(\Da)$. 
\end{theorem}

\begin{proof}
As has been said before, for $f\in \Phi(\Qs)$,  $u(x,t)=Z(x,t) * f(x)$ is   a classical solution to the heat equation (\ref{HeatEquation}). Also, $u(x,t)\in \Phi(\Qs)\subset \Dom(D^{\alpha})$ for any $t\geq 0$. Since $\Phi(\Qs)$ is dense in $L^2(\Qs)$, the operator $\Ss(t)=\Tt(t)$ for each $t\geq 0$ and the function $u(x,t)=Z(x,t) * f(x)$ is a solution of the Cauchy problem for any $f\in \Dom(\Da)$.
\end{proof}

\section[The Markov process on $\Qs$]{The Markov process on $\Qs$}
\label{markov_process}

In this section the fundamental solution of the heat equation  $Z(x,t)$  is shown to be the transition density function of a Markov process  on $\Qs$ (see \cite{Dyn} and also \cite{Tai}). 

Let $\alpha$ be a positive real number and denote by $\B$ the Borel $\sigma$--algebra of $\Qs$. Define 
\[ p(t,x,y) := Z(x-y,t) \quad (t>0, x,y\in \Qs) \]
and
$$ P(t,x,y) =
\begin{cases}
\int_B p(t,x,y) dy & \text{if}\; t>0, x\in \Qs, B\in \B,\\
\Delta_B(x) & \text{if}\; t=0.
\end{cases}
$$

From Theorem \ref{heat-distribution},  it  follows that  $p(t,x,y)$ is a normal transition density and  $P(t,x,B)$ is a normal transition function on $\Qs$ (see \cite{Dyn}, Section 2.1 for further details).

\begin{lemma}
\label{condition-L-M}
The transition function $P(t,y,B)$ satisfies the following two conditions:
\begin{enumerate}[a.]
\item For each $s\geq 0$ and any compact subset $B$ of $\Qs$
\[ \lim_{x\to \infty} \sup_{t\leq s} P(t,x,B) = 0 \quad  (\text{Condition}\; L(B)). \]

\item For each $\epsilon>0$ and any compact subset $B$ of $\Qs$
\[ \lim_{t\to 0+} \sup_{x\in B} P(t,x,\Qs\setminus B) = 0 \quad  (\text{Condition}\;  M(B)). \]
\end{enumerate}
\end{lemma}

\begin{proof}
Since $B$ is compact, $d(x) := dist(x,B)\To \infty$ as $x\to \infty$. From Lemma 
\ref{heat-estimatesII} it follows $Z(x-y,t)\leq s\norm{d(x)}^{-\alpha-1}$ for any $y\in B$ and $t\leq s$. Hence 
\[ P(t,x,B)\leq s\norm{d(x)}^{-\alpha-1} \mu(B)\To 0 \] 
as $x\to \infty$. This implies Condition $L(B)$.

To verify Condition $M(B)$ we proceed as follows: for $y\in \Qs\setminus B$ we get $\norm{x-y}>\epsilon$. The statement follows from Proposition \ref{heat-estimates}:
$$
P(t,x,\Qs\setminus B)\leq Ct \To 0,\;  t\to 0^{+}.
$$
\end{proof}

\begin{theorem} 
The heat kernel $Z(x,t)$ is the transition density of a time and space homogeneous Markov process which is bounded, right--continuous and has no discontinuities other than jumps.
\end{theorem}

\begin{proof}
The result follows from Lemma \ref{condition-L-M}  and the fact that $\Qs$ is a second countable and locally compact  ultrametric space (see \cite{Dyn},Theorem 3.6). 
\end{proof}

\section[Cauchy problem for parabolic type equations on $\Qs$]{Cauchy problem for parabolic type equations on $\Qs$}
\label{Cauchy_problem}

In this section two other classical formulations of the Cauchy problem on $\Qs$ are described. 

\subsection[Homogeneous equation with values in $L^2$]{Homogeneous equation with values in $L^2$} 

Recall that a $C_0$ semigroup $\exp(tL)$, with infinitesimal generator $L$, is smooth if for any $t>0$ and any $f \in L^2(\Qs)$, the element $\exp(tL)f$ of $L^2(\Qs)$ belongs to the domain of $L$.

\begin{theorem}
\label{solution_heatequation} 
The $C_0$ semigroup $S(t)$ is smoothing. In other words, if $f$ is any square integrable function on $\Qs$,  the Cauchy problem
\begin{equation*}
\begin{cases}
\frac{{\partial}u(x,t)}{{\partial t}} + \Da u(x,t) = 0,  \ x\in \Qs, \ t> 0,  \\
u(x,0) = f(x),
\end{cases}
\end{equation*} 
has a unique solution
\[ u(x,t) = \int_{\Qs} f(x - y) Z(y,t) dy. \]
\end{theorem}

\begin{proof} 
From Proposition \ref{convergence}, it follows that the $C_0$ semigroup given by
$$ f(\xi)\mTo  \exp(t\norm{\xi}^\alpha)f(\xi) \qquad(t \geq 0),$$ 
is smooth. Since the Fourier transform is unitary, the  $C_0$ semigroup $\Ss(t)$ is also smooth.
\end{proof}

\subsection[Non--homogeneous equations]{Non--homogeneous equations}

Consider the following Cauchy problem
$$
\begin{cases}
\frac{\partial u(x,t)}{\partial t} + \Da u(x,t) = f(x,t), & x\in \Qs, t\in [0,T], T>0,\\
u(x,0)=u_0(x) & u_0\in \Dom(\Da).
\end{cases}
$$

The Duhamel's Principle states:

\begin{theorem} 
Let $\alpha>0$ and let $f(x,\cdot)\in \Co([0,T],L^2(\Qs))$. Assume that at least one of the following conditions is satisfied:

\begin{enumerate}
\item $f(x,\cdot) \in L^1((0,T),\Dom(\Da))$;
\item $f(x,\cdot) \in W^{1,1}((0,T),L^2(\Qs))$.
\end{enumerate}

Then the Cauchy problem has a unique solution given by
\[ 
u(x,t) = \int_{\Qs} Z(x-y,t) u_0(y) dy + 
\int_0^t\left\{ \int_{\Qs} Z(x-y,t-\tau)f(y,\tau) dy \right\}d\tau, 
\]
where $d\tau$ is the Lebesgue measure.
\end{theorem}

\section[Final remarks]{Final remarks}
\label{final_remarks}

The theory developed in this article can be applied to more general
Abelian topological groups. More explicitly, let $G$ be a selfdual, second countable, totally disconnected, locally compact, Abelian topological group with a filtration by compact and open subgroups $\{H_n\}_{n\in \Z}$,
\[ \{0\}\subset \cdots \subset H_{-n} \subset \cdots \subset H_0 \subset \cdots \subset H_n \subset \cdots \subset G, \]
such that:

\begin{enumerate}
\item  $H_0=H$ is a fixed compact and open subgroup of $G$, 
 \item  the index of the quotients $H_{n+1}/H_n$ is uniformly bounded from above,
 
 \item  the annihilator, $\mathrm{Ann}_G(H_n)$,  satisfies 
\[ \mathrm{Ann}_G(H_n) = H_{-n}, \] 
for all $n\in \Z$, and
\item the following relations are  satisfied: 
\begin{equation*}
\label{union-interseccionQl}
\bigcap_{n\in \Z} H_n = \{0\}\quad \text{ and } \quad \bigcup_{n\in \Z} H_n = G.
\end{equation*}

 \end{enumerate}

Notice that, since $G$ is autodual, there exist an isomorfism of $G$ with its Pontryagin dual group $\widehat{G}$, i.e. 
a function $\xi \mapsto \chi_{\xi}$ which identifies $G$ and $\widehat{G}$ as topological groups.  This identification makes  expression $\mathrm{Ann}_{G}(H_n) = H_{-n}$, in property $(3)$ above, meaningful. 

Normalize the Haar measure $\mu$ on $G$ in such a way that $\mu(H)=1$. This implies that the measure of any other subgroup $H_n$ of $G$ is given either by the index $[H:H_n]$ or by the index $[H_n:H]$. The group $G$ has a unique  $G$-invariant  ultrametric $d_G$ such that the balls centred at zero coincide with the elements of the filtration $\{H_n\}_{n\in \Z}$ and the radius of any ball equals to its Haar measure.  Let $\lambda_n$ be the radius and Haar measure of $H_n$.

The topological and algebraic properties  of $G$ are expressed by the projective and inductive limits:
$$
H_0=\varprojlim _{ n \leq 0} H_0/H_{n} \qquad G=  \varinjlim_{ n \geq 0}  H_n.
$$ To the ultrametric spaces, $(H,d_H)$ and  $(G,d)$, there corresponds, respectively,  a tree $\mathcal{T}(H)$ with finite ramification index, and a tree $\mathcal{T}(G)$ whose endspace is identified with 
$G$. Consequently, $L^2(G)$ is a separable Hilbert space (see \cite{CE1}, \cite{KK1} and  \cite{KK2}).

In addition, the topology of the locally constant functions  of compact support $\D(\Qs)$ is expressed by the inductive limits
  $$
  \D^{\ell}(G) = \varinjlim_{k} \D^{\ell}_k(G) \quad \text{ and } \quad  
  \D(G) =\varinjlim_{\ell} \D^{\ell}(G) 
  $$ and $\D(\Qs)$ is a locally convex complete topological algebra and a nuclear space.
  
We have the Fourier transform $
\F(f)(\xi)=\int_G f(x)\chi_{\xi}(x) d\mu(x)$ and due to   equality $\Ann_{\Qs} (H_n)=H_{-n}$, it 
satisfies  $ \F: \D^{\ell}_k(\Qs)\To \D^{-k}_{-\ell}(\Qs)$  and  gives an isomorfism $\D(\Qs)\cong \D(\Qs)$ of locally  convex topological linear space. Furthermore, the Fourier transform $ \F: L^2(G) \To L^2(G)$
is an isometry of Hilbert spaces.

For any $\alpha >  0$, the  function $\norm{\cdot}_{G}^{\alpha}$ related to the ultrametric $d_G$ defines the pseudodifferential operator
$D^\alpha: \Dom(D^\alpha) \subset L^2(G) \To L^2(G)$ given by 
$$ D^\alpha(f)= \F^{-1}_{\xi \to x} \big(\norm{\xi}_{G} ^\alpha \F_{x \to \xi}(f) \big)  \qquad (f \in \Dom(D^\alpha) ).$$
The operator $-D^\alpha$ is a positive selfadjoint unbounded operator with spectrum $\{ 0 \} \cup \{ \lambda_n^\alpha \}_{n \in \Z}   $.
The heat kernel $Z(x,t)= \F^{-1}_{\xi \to x}\big(\exp(-t \norm{\xi}^\alpha)\big)$ is a well defined positive function, given by a  formula similar to the one  in Proposition  \ref{freeHeatKernel} and  satisfies the estimate of Proposition \ref{heat-estimates}. 

As a consequence of all that have been said, the following result holds.

\begin{theorem}\label{generalization}
  If $f$ belongs to  $\Dom(-\Da) \subset L^2(G)$, the Cauchy problem
\begin{equation*}
\begin{cases}
\frac{{\partial}u(x,t)}{{\partial t}} + \Da u(x,t) = 0,  \ x \in \Qs, \ t \geq 0,  \\
u(x,0) = f(x),
\end{cases}
\end{equation*} 
has a classical solution $u(x,t)$ determined by the convolution of $f$ with the heat kernel $Z(x,t)$. In addition, $Z(x,t)$ is the transition density of a time and space homogeneous Markov process which is bounded, right--continuous and has no discontinuities other than jumps.
\end{theorem}

Examples of these groups are detailed as follows:
\begin{enumerate}
\item For any fixed positive integer number $n$, let $G$ be the set of $n$--adic numbers $\Q_n$ with $H_0=\Z_n$ the maximal compact and open subring. The filtration is given by 
$H_{\ell} = n^{\ell} \Z_n$, for $\ell\in \Z$. 

\item For any prime number $p$ and a positive integer number $n$, let $G=(\Qp^n,+)$ with  
$H_0=\Zp^n$ the maximal compact and open subring. The filtration is given by 
$H_{\ell} = p^{\ell} \Z_p^n$, for $\ell\in \Z$. The metric here induced the Taibleson operator on $\Qp^n$ with a normalizing factor.

\item Fix any prime number $p$ and set $G$ the restricted direct product of a countable copies of the fixed field $\Qp$ with respect to $\Zp$. That is, $G=\prod'_{n\geq 1} \Qp$ and $H_0=\prod_{n\geq 1} \Zp$. The members of the filtration can be given as:
\begin{enumerate}[$\bullet$]
\item $H_0 = \prod_{n\geq 1} \Zp$,
\item $H_{\pm 1} = p^{\mp 1}\Zp \times \Zp\times \cdots$,
\item $H_{\pm 2} = p^{\mp 1}\Zp \times p^{\mp 1}\Zp \times \Zp\times \cdots$,
\item $H_{\pm 3} = p^{\mp 2}\Zp \times p^{\mp 1}\Zp \times \Zp\times \cdots$  
\item $H_{\pm 4} = p^{\mp 2}\Zp \times p^{\mp 1}\Zp \times p^{\mp 1}\Zp \times \Zp\times \cdots$  
\end{enumerate}

\item Finite products and restricted direct products of the examples above provides a large class of groups satisfying the requirements.
\end{enumerate}

Last but not least,  the finite \`adele ring $\Af$ is a selfdual, second contable, and  totally disconnected topological ring where one can take $H_0=\prod_{p \in \Pp} \Zp$. However, even though there are many  filtrations  $\{H_n\}_{n\in \Z}$ of $\Af$, satisfying properties (1),(3),(4), any filtration satisfying those properties does not 
satisfies property (2) (see \cite{CE1}).


\begin{thebibliography}{BGPW}

\bibitem[AKS]{AKS} S. Albeverio; A.Y. Khrennikov and V.M. Shelkovich. \emph{Theory of $p$--Adic Distributions}. LMS, Lectures Notes Series 370, Cambridge University Press, New York, 2010.

\bibitem[Apo]{Apo} T.M. Apostol. \emph{Introduction to Analytic Number Theory}. Undergraduate Texts in Mathematics. Springer-Verlag, New York--Heidelberg, 1976.

\bibitem[BGPW]{BGPW} A.D. Bendikov, A.A. Grigor'yan, Ch. Pittet and W. Woess. \emph{Isotropic Markov Semigroups on Ultrametric Spaces}. Russian Math. Surveys,  \textbf{69}, no. 4, pp 589--680, 2014.

\bibitem[Bro]{Bro} K.A. Broughan. \emph{Adic Topologies for the Rational Integers}. Canad. J. Math. \textbf{55}, no. 4, 71--723, 2003. 

\bibitem[CE1]{CE1} M. Cruz--L\'opez and S. Estala--Arias. \emph{Invariant Ultrametrics and Markov Processes on the Finite Ad\`ele Ring of $\Q$}. $p$--Adic Numbers, Ultrametric Analysis and Applications, \textbf{8}, no. 2, pp. 89--114, 2016.

\bibitem[CE2]{CE2} M. Cruz--L\'opez and S. Estala--Arias. \emph{Fourier Analysis on the Ad\'ele Ring of $\Q$}. Submmited for publishing.  


\bibitem[Dyn]{Dyn} E.B. Dynkin. \emph{Theory of Markov Processes}. Translated from the Russian by D. E. Brown and edited by T. Köváry. Reprint of the 1961 English translation. Dover Publications, Inc., Mineola, NY, 2006.

\bibitem[EN]{EN} K.J. Engel and R. Nagel.\emph{One--Parameter Semigroups for Linear Evolution Equations}. Springer--Verlag, 2000.

\bibitem[Igu]{Igu} J.I. Igusa. \emph{An Introduction to the Theory of Local Zeta Functions}. 
AMS/IP Studies in Advanced Mathematics, \textbf{14}. American Mathematical Society, Providence, RI; International Press, Cambridge, MA, 2000.

\bibitem[KK1]{KK1} A.Y. Khrennikov and S.V. Kozyrev. \emph{Wavelets on Ultrametric Spaces}. Appl. Comput. Harmon. Anal. \textbf{19}, no. 1, 61--76, 2005.

\bibitem[KK2]{KK2} S.V. Kozyrev and A.Y. Khrennikov. \emph{Pseudodifferential Operators on Ultrametric Spaces, and Ultrametric Wavelets}. (Russian) Izv. Ross. Akad. Nauk Ser. Mat. \textbf{69} (2005), no. 5, 133--148; translation in Izv. Math. \textbf{69}, no. 5, 989--1003, 2005.

\bibitem[KR]{KR} A.Y. Khrennikov and Y.V. Radyno. \emph{On Adelic Analogue of Laplacian}. Proc. Jangjeon Math. Soc. \textbf{6}, no. 1, 1--18, 2003.

\bibitem[Koc]{Koc} A.N. Kochubei. \emph{Pseudo--differential Equations and Stochastics over Non--Archimedean Fields}. Monographs and Textbooks in Pure and Applied Mathematics, \textbf{244}. Marcel Dekker, Inc., New York, 2001. 

\bibitem[KKS]{KKS} A.V. Kosyak; A.Y. Khrennikov and V.M. Shelkovich. \emph{Pseudodifferential Operators on Ad\`eles and Wavelet Bases}. (Russian) Dokl. Akad. Nauk \textbf{444} (2012), no. 3, 253--257; translation in Dokl. Math. \textbf{85}, no. 3, 358--362, 2012.

\bibitem[Paz]{Paz} A. Pazy, \emph{Semigroups of Linear Operators and Applications to Partial Differential Equations}, Springer - Verlag, 1983. 

\bibitem[Sch]{Sch} H.H. Schaefer, \emph{Topological Vector Spaces}, Springer - Verlag, 1971.

\bibitem[Tai]{Tai} K. Taira, \emph{Boundary Value Problems and Markov Processes}, Springer - Verlag, 2009.

\bibitem[TZ]{TZ} S.M. Torba and W. Z\'u\~niga--Galindo. \emph{Parabolic Type Equations and Markov Stochastic Processes on Adeles}. J. Fourier Anal. Appl. \textbf{19}, no. 4, 792--835,  2013. 

\bibitem[VVZ]{VVZ} V.S. Vladimirov; I.V. Volovich and E.I. Zelenov. \emph{p--Adic Analysis and Mathematical Physics}. Series on Soviet and East European Mathematics, \textbf{1}. World Scientific Publishing Co., Inc., River Edge, NJ, 1994.

\bibitem[Zun]{Zun} W. A. Z\'u\~niga--Galindo, \emph{Parabolic Equations and Markov Processes over p--Adic Fields}. Potential Anal. \textbf{28}, no. 2, 185--200, 2008.


\bibitem[Zun1]{Zun1} W. A. Z\'u\~niga--Galindo, \emph{Pseudodifferential Equations Over Non-Archimedean Spaces}. Lecture Notes in Mathematics. Springer International Publishing, 2016.


\end{thebibliography}
\end{document}